\documentclass[11pt]{article}
\usepackage{mathrsfs, amsthm, amssymb, graphicx, amsmath, Tabbing}

\theoremstyle{plain}
\newtheorem{Conjecture}{Conjecture}
\newtheorem{Theorem}{Theorem}
\newtheorem{Lemma}{Lemma}
\newtheorem{Corollary}{Corollary}

\newcommand{\N}{\mathbb{N}}

\newcommand{\F}{\mathscr{F}}

\newcommand{\G}{\mathscr{G}}

\title{Minimum density of union-closed families}
\author{Igor Balla}
\date{\today}

\begin{document}

\maketitle

\begin{abstract}
Let $\F$ be a finite union-closed family of sets whose largest set contains $n$ elements. In \cite{Wojcik92}, W\'{o}jcik defined the density of $\F$ to be the ratio of the average set size of $\F$ to $n$ and conjectured that the minimum density over all union-closed families whose largest set contains $n$ elements is $(1 + o(1))\log_2{n}/(2n)$ as $n \rightarrow \infty$. We use a result of Reimer \cite{Reimer03} to show that the density of $\F$ is always at least $\log_2{n}/(2n)$, verifying W\'{o}jcik's conjecture. As a corollary we show that for $n \geq 16$, some element must appear in at least $\sqrt{(\log_2{n})/n}(|\F|/2)$ sets of $\F$.
\end{abstract}

\section{Preliminaries and Notations}
Given a family of sets $\F$, we say $\F$ is union-closed if for all $A, B \in \F$, $A \cup B \in \F$. In what follows, a union-closed family will always be taken to mean a finite union-closed family of finite sets. Let $\N = \{1, 2, 3, \ldots\}$ be the set of natural numbers. We denote the cardinality of a finite set $A$ by $|A| := \sum_{x \in A}{1}$ and denote the union of a family of sets $\F$ by $\bigcup{\F} := \bigcup_{A \in \F}{A}$. Given sets $A,B$, the set difference $A \setminus B := \{x \in A | x \notin B \}$.

Let $\F$ be a union-closed family and let $n = |\bigcup{\F}|$. To avoid trivial cases, in this paper we only consider $\F$ with $n \geq 1$. Define $\F_a = \{S \in \F | a \in S\}$ for all $a \in \bigcup{\F}$. We define the density of $\F$ by 
\[ D(\F) := \frac{1}{n |\F|}\sum_{A \in \F}{|A|} = \frac{1}{n |\F|}\sum_{a \in \bigcup{\F}}{|\F_a|}\]
and denote the minimum density over all union-closed families $\F$ with $|\bigcup{\F}| = n$ by $s_n := \min{\{D(\F) | \F \text{ union-closed, } |\bigcup{\F}| = n\}}$ for all $n \in \N$.

\section{Some Previous Conjectures and Results}
Interest in studying the structure of union-closed families first arose because of the following conjecture due to Frankl in 1979.

\begin{Conjecture}[Frankl]
For all union-closed families $\F$, there exists an $a \in \bigcup{\F}$ such that $|\F_a| \geq |\F|/2$.
\end{Conjecture}

Although much research has been done on union-closed families, we still seem to be far from being able to prove Conjecture 1. Roberts \cite{Roberts92} showed that Conjecture 1 holds when $|\F| \leq 40$ and Bo\v{s}njak and Markovi\'{c} \cite{Markovic08} showed that it holds when $|\bigcup{\F}| \leq 11$. Additionally, Sarvate and Renaud, Poonen \cite{Sarvate89, Poonen92} and others have shown that Conjecture 1 holds if there exists an $S \in \F$ with $|S| = 1,2$. In order to generalize this idea, W\'{o}jcik \cite{Wojcik92} defined the notions of density and minimum density as stated above and proved the following theorem.

\begin{Theorem}[W\'{o}jcik \cite{Wojcik92}]
Let $\F$ be a union-closed family and let $S \in \F$ with $|S| = k$. Then
\[ \sum_{a \in S}{|\F_a|} \geq k s_k |\F|.\]
\end{Theorem}

One can easily verify that $s_1, s_2 = 1/2$, so Theorem 1 shows that Conjecture 1 holds if there exists $S \in \F$ with $|S| = 1,2$. Unfortunately, $s_3 = 4/9$ and W\'{o}jcik showed in \cite{Wojcik92} that for any $\epsilon > 0, k \in \N$, one can find a union-closed family $\F$ with $S \in \F$, $|S| = k$ such that $|\F_a| < (s_k + \epsilon)|\F|$ for all $a \in S$. Thus having a 3-set in a union-closed family $\F$ is not enough to guarantee that one of its elements satisfies Conjecture 1. In any case, it is an interesting question to be able to determine $s_k$ in general. W\'{o}jcik conjectured that

\begin{Conjecture}[W\'{o}jcik \cite{Wojcik92}]
For all $n \in \N$, a union-closed family $\F$ with $|\bigcup{\F}| = n$ which attains the minimum density $s_n$ is of the form $\F := \{A | A \subseteq \{1, 2, \ldots, k\} \} \cup \{ \{1, 2, \ldots, n\} \}$ where $k = \lfloor\log_2{n}\rfloor$ or $\lceil \log_2{n}\rceil$.
\end{Conjecture}

He verified this claim for $n \leq 10$, and noted that as a consequence we would have $s_n = (1 + o(1))\frac{\log_2{n}}{2n}$ as $n \rightarrow \infty$. W\'{o}jcik claimed, without providing the proof, that Conjecture 1 implies this asymptotic result. In Corollary 1, we verify this asympotic result, giving more evidence towards the truth of Conjectures 1 and 2. In order to do so, we use a theorem proven by Reimer in \cite{Reimer03} bounding the average set size of a union-closed family.

\begin{Theorem}[Reimer \cite{Reimer03}]
For all union-closed families $\F$,
\[ \frac{1}{|\F|}\sum_{A \in \F}{|A|} \geq \frac{1}{2}\log_2{|\F|}.\]
\end{Theorem}

\section{Main Results}

The main goal of this paper to show that for all union-closed families $\F$ with $|\bigcup{\F}| = n$, $D(\F) \geq \log_2{n}/(2n)$. In the case where $|\F| \geq n$, Theorem 2 gives us the desired result. In the other case $|\F| < n$, we will need Lemma 2.

\begin{Lemma}
Let $\F$ be a union-closed family, let $A \in \F$ be a minimal non-empty set and let $\G := \F \setminus \{A\}$. If there exist $a,b \in \bigcup{\G}$ such that $\G_a = \G_b$ but $\F_a \neq \F_b$ then $a \in B$ for all non-empty $B \in \F$ or $b \in B$ for all non-empty $B \in \F$. 
\end{Lemma}
\begin{proof}
If both $a,b \in A$ or both $a,b \notin A$ then $\F_a = \F_b$. Otherwise without loss of generality we have $a \in A$ and $b \notin A$. Now suppose there exists a non-empty $B \in \G$ such that $a \notin B$. Then $B \notin \G_a = \G_b$, so $b \notin B$ and thus $b \notin A \cup B$. $A \cup B \neq A$ by minimality of $A$, and so $A \cup B \in \G$ since $\F$ is union-closed. But $a \in A \subseteq A \cup B$, contradicting $\G_a = \G_b$.
\end{proof}

\begin{Lemma}
For all union-closed family $\F$ with $|\bigcup{\F}| \geq 2$ and $|\F| < |\bigcup{\F}|$, there exist distinct $a,b \in \bigcup{\F}$ such that $\F_a = \F_b$.
\end{Lemma}
\begin{proof}
We induct on $|\F|$. For $|\F| = 1$, $\F = \{ \bigcup{\F} \}$, so we can choose any $a,b \in \bigcup{\F}$ to satisfy $\F_a = \F_b$. Now let $\F$ be a union-closed family with $|\F| = m \geq 2$ and $|\bigcup{\F}| = n$ such that $m < n$. Let $A \in \F$ be a minimal non-empty set and let $\G := \F \setminus \{A\}$. If $A = \bigcup{\F}$ then we must have $\F = \{A, \emptyset\}$, so we can choose any $a,b \in A$ to satisfy $\F_a = \F_b$. Otherwise we have $|\bigcup{\G}| = |\bigcup{\F}| = n$. No 2 sets of $\G$ can have union $A$ by minimality of $A$, so $\G$ is union-closed because $\F$ is. Since $|\G| = m - 1$, we can apply the induction hypothesis to $\G$ to obtain distinct $a,b \in \bigcup{\G}$ such that $\G_a = \G_b$. If $\F_a = \F_b$ then we are done. Otherwise we can apply Lemma 1 and without loss of generality have $a \in B$ for all non-empty $B \in \F$. Now define $\G' := \{B \setminus \{a\} | B \in \G\}$, so that $|\G'| \leq |\G| = m - 1$ and $|\bigcup{\G'}| = |\bigcup{\G} \setminus \{a\}| = n - 1$. $\G'$ is union-closed because $\G$ is, so we can apply the induction hypothesis to $\G'$ to obtain distinct $c,d \in \bigcup{\G'}$ such that $\G'_c = \G'_d$. Thus $\G_c = \G_d$. If $\F_c = \F_d$ then we are done. Otherwise again applying Lemma 1, without loss of generality we have $c \in B$ for all non-empty $B \in \F$. So then $\F_a = \F_c$, and the claim holds by induction.
\end{proof}

Note that for any $n \in \N$, the family $\F := \{ \{1, 2, \ldots k\} | k \in \{1, 2, \ldots, n \}\}$ is union-closed with $|\F| = n = |\bigcup{\F}|$ such that $\F_a \neq \F_b$ for all distinct $a,b \in \bigcup{\F}$. Thus Lemma 2 is in some sense best possible.

\begin{Theorem}
For all $n \in \N$,
\[ s_n \geq \frac{\log_2{n}}{2n}.\]
\end{Theorem}
\begin{proof}
We proceed by induction. For $n=1$, trivially $s_1 \geq 0 = \log_2{n}/(2n)$. Now suppose $s_n \geq \log_2{n}/(2n)$ for some $n \in \N$ and let $\F$ be a union-closed family with $|\bigcup{\F}| = n + 1$. If $|\F| \geq n + 1$ then by Theorem 2,
\[ D(\F) = \frac{1}{(n + 1) |\F|}\sum_{A \in \F}{|A|} \geq \frac{\log_2{|\F|}}{2(n + 1)} \geq \frac{\log_2{(n+1)}}{2(n + 1)}. \]
Otherwise $|\F| \leq n$, so we can apply Lemma 2 to obtain distinct $a,b \in \bigcup{\F}$ such that $\F_a = \F_b$. Thus we can define $\G := \{A \setminus \{a\} | A \in \F \}$, so that $|\G| = |\F|$ and $\G$ is union-closed since $\F$ is. Now $|\bigcup{\G}| = |\bigcup{\F} \setminus \{a\}| = n$, so we can apply the induction hypothesis to obtain
\[ \sum_{x \in \bigcup{\F} \setminus \{a\}}{|\F_x|} = \sum_{x \in \bigcup{\G}}{|\G_x|} = n D(\G)|\G| \geq n s_n |\G| \geq \frac{1}{2}|\F|\log_2{n}. \]
Using the limit definition of $e$ we see that
\[ \left(\frac{n+1}{n}\right)^{|\F|} \leq \left(1 + \frac{1}{n}\right)^n \leq e \leq 4.\]
Thus 
\[ |\F|(\log_2{(n+1)} - \log_2{n}) = \log{\left ( \left(\frac{n+1}{n}\right)^{|\F|} \right )} \leq \log_2{4} = 2. \]
Finally observing that $|\F_a| \geq 1$, we obtain
\[ \sum_{x \in \bigcup{\F}}{|\F_x|} = |\F_a| + \sum_{x \in \bigcup{\F} \setminus \{a\}}{|\F_x|} \geq 1 + \frac{1}{2}|\F|\log_2{n} \geq \frac{1}{2}|\F| \log_2{(n+1)}.\]
Thus $D(\F) \geq \log_2{(n+1)}/(2(n+1))$ and the claim holds by induction.
\end{proof}

With repeated application of Lemma 2, it can be proven that either $\frac{1}{|\F|}\sum_{A \in \F}{|A|} \geq n(n-1)/2$ or there exist $\alpha, \beta \in \bigcup{\F}$ distinct, such that $\F_{\alpha} = \F_{\beta}$. This may be combined with Theorem 2, to yield a better lower bound on $s_n$ using the same techniques as in Theorem 3. Work is currently being done along these lines to verify Conjecture 2.

\begin{Corollary}
\[ s_n = (1 + o(1))\frac{\log_2{n}}{2n}. \]
as $n \rightarrow \infty$.
\end{Corollary}
\begin{proof}
Let $n \in \N$ and let $k := \lceil\log_2{n}\rceil$. Define $\F := \{A | A \subseteq \{1, \ldots, k\} \} \cup \{ \{1, 2, \ldots, n\} \}$. Observe that $\F$ is union-closed with $|\bigcup{\F}| = n$. Thus we have
\[ s_n \leq D(\F) = \frac{k 2^{k-1} + n}{n(2^k + 1)} = (1 + o(1))\frac{\log_2{n}}{2n}\]
as $n \rightarrow \infty$. On the other hand, $s_n \geq \log_2{n}/(2n)$ by Theorem 3.
\end{proof}

\begin{Corollary}
Let $n \in \N$ with $n \geq 16$ and let $\F$ be a union-closed family with $|\bigcup{\F}| = n$. Then there exists an $a \in \bigcup{\F}$ such that 
\[ |\F_a| \geq \frac{1}{2}\sqrt{\frac{\log_2{n}}{n}}|\F|. \]
\end{Corollary}
\begin{proof}
Let $S \in \F$ be a non-empty set with minimum $|S|$ and let $k =|S|$. Define $\G := \F \setminus \{\emptyset, \bigcup{\F}\}$ so that $\sum_{A \in \G}{|A|} \geq k(|\F| - 2)$, and thus
\[ \max_{a \in \bigcup{\F}}{|\F_a|} \geq \frac{1}{n}\sum_{a \in \bigcup{\F}}{|\F_a|} = \frac{1}{n}\sum_{A \in \F}{|A|} = 1 + \frac{1}{n}\sum_{A \in \G}{|A|} \geq 1 + \frac{k}{n}(|\F|-2).\]
If $k > n/2$, then
\[ \max_{a \in \bigcup{\F}}{|\F_a|} \geq 1 + \frac{k}{n}(|\F| - 2) > \frac{|\F|}{2} \geq \frac{1}{2}\sqrt{\frac{\log_2{n}}{n}}|\F| \]
so we are done. Otherwise $k \leq n/2$ and so we obtain
\begin{equation}
\max_{a \in \bigcup{\F}}{|\F_a|} \geq 1 + \frac{k}{n}(|\F| - 2) \geq \frac{2k + (|\F| - 2)k}{n} = \frac{k}{n}|\F|.
\label{ineq1}
\end{equation}
Additionally, we can apply Theorem 1 to obtain
\[ \max_{a \in \bigcup{\F}}{|\F_a|} \geq \max_{a \in S}{|\F_a|} \geq \frac{1}{k}\sum_{a \in S}{|\F_a|} \geq s_k |\F|. \]
If $k \leq 2$, then we observe that $s_k = 1/2$ so we are done. Otherwise $k \geq 3$. By Theorem 3 we have
\begin{equation}
\max_{a \in \bigcup{\F}}{|\F_a|} \geq s_k |\F| \geq \frac{\log_2{k}}{2k}|\F|.
\label{ineq2}
\end{equation}
Consider the function $f:[4, \infty) \rightarrow [16, \infty)$ defined by $f(x) := 2x^2/\log_2{x}$ for all $x \in [4, \infty)$. Since $f$ increasing and $f(4) = 16$, we can define the inverse function $g:[16, \infty) \rightarrow [4, \infty)$ of $f$. We have
\[ \frac{g(n)}{n} = \frac{g(n)}{f(g(n))} = \frac{\log_2{g(n)}}{2 g(n)}. \]
Thus $g(n) = \sqrt{n \log_2{\sqrt{g(n)}}}$ and since $g(n) \geq 4$,
\begin{align*}
\frac{g(n)}{n} &= \frac{\sqrt{n \log_2{\sqrt{g(n)}}}}{n} = \sqrt{\frac{\log_2{g(n)}}{2n}} = \sqrt{\frac{\log_2{\sqrt{n \log_2{\sqrt{g(n)}}}}}{2n}}\\
							 &= \sqrt{\frac{\log_2{n} + \log_2{(\log_2{\sqrt{g(n)}})}}{4n}} \geq \frac{1}{2}\sqrt{\frac{\log_2{n}}{n}}.
\end{align*}
Now suppose for sake of contradiction that $\max_{a \in \bigcup{\F}}{|\F_a|} < (g(n)/n)|\F|$. Then by (\ref{ineq1}), $g(n) > k$. But by (\ref{ineq2}) we also have 
\[ \frac{\log_2{g(n)}}{2g(n)} = \frac{g(n)}{n} > \max_{a \in \bigcup{\F}}{|\F_a|} \geq \frac{\log_2{k}}{2k}. \]
Since $\log_2{x}/(2x)$ is decreasing for $x > e$ and $g(n) \geq 4 > e$, $k \geq 3 > e$, it follows that $g(n) < k$, a contradiction. Thus
\[ \max_{a \in \bigcup{\F}}{|\F_a|} \geq \frac{g(n)}{n}|\F| \geq \frac{1}{2}\sqrt{\frac{\log_2{n}}{n}}|\F|. \qedhere\]
\end{proof}

Corollary 2 clearly does not give a sharp lower bound and with some work the technique used in the proof can be extended to prove a slightly better bound. In this paper we decided against spoiling the neatness of the result by doing so. In any case, the above technique is not powerful enough to prove Conjecture 1, since the minimum density $s_n < 1/2$ for $n \geq 3$.

\bibliography{UCS}

\end{document}